\documentclass[11pt]{article}

\title{Staircase palindromic polynomials}

\author{Rabi K.C.\footnotemark[1] \footnotemark [2], Abdalnaser Algoud\footnotemark[1]}

\usepackage{tikz}
\usepackage{amsmath, mathtools, amsthm}
\usepackage{epsfig}
\usepackage{tikz}
\usepackage{pgfplots}
\usepackage{graphicx}
\graphicspath{{../images/}}
\usepackage{amsfonts}
\usepackage{amssymb}
\usepackage{verbatim}
\usepackage{caption}
\usepackage{subcaption}
\usepackage{color}

\numberwithin{equation}{section}
\newtheorem{defn}{Definition}[section]

\newtheorem{thm}[defn]{Theorem}

\newtheorem{lem}[defn]{Lemma}
\newtheorem{prop}[defn]{Proposition}

\newtheorem{rem}[defn]{Remark}
\newtheorem{ex}{Example}

\newcommand{\D}[1]{{\mathbb#1}}%Doubled -Blackboard bold - caps only

\newcommand{\CC}{{\D{C}}}

\begin{document}

\maketitle

\begin{abstract}
We study a class of monic-palindromic polynomials that we call staircase palindromic polynomials. Specifically, suppose  $S(x,n,h)$ is a polynomial of degree~$n$ with the special form:
\begin{equation*}
\begin{split}
 S(x,n,h) & =   x^n + 2 x^{n-1} + 3 x^{n-2} +  \ldots  + (h-1) x^{n-h+2} \\
         & \hspace*{1cm} + h x^{n-h+1}  +  \ldots  + h x^{h-1} \\
          & \hspace*{2cm}     + (h-1) x^{h-2} + \ldots + 2 x + 1.
\end{split}
\end{equation*} 
Then $S(x,n,h)$ can be factored as a product of cyclotomic polynomials. Moreover, for any given n, there are $ \lceil{\frac{n+1}{2}\rceil}$ staircase polynomials, all of whose factors can be derived using two parameter~$n$ and~$h$ with the help of cyclotomic polynomials. After that we explore some classes of polynomials that can be converted to staircase polynomials. 
\end{abstract}

%%%%%%%%%%%%%%%%%%%%%%%%%%%%%%%%%%%%%%%%%%%%%%%%%%%%%%%%%%%%%%%%%%
%%%%%%%%%%%%%%%%%%%%%%%%%%%%%%%%%%%%%%%%%%%%%%%%%%%%%%%%%%%%%%%%%%
%%%%%%%%%%%%%%%%%%%%%%%%%%%%%%%%%%%%%%%%%%%%%%%%%%%%%%%%%%%%%%%%%%
%%%%%%%%%%%%%%%%%%%%%%%%%%%%%%%%%%%%%%%%%%%%%%%%%%%%%%%%%%%%%%%%%%

\section{Introduction}
%Finding the roots of any given polynomial is generally a difficult task.  However, in this paper we focus on a class of palindromic polynomials called  “staircase palindromic polynomials,” and derive the formula for their structure. Because of the symmetric property that staircase palindromic polynomials have, we are able to reduce them into the product  of cyclotomic polynomials. Moreover, in some cases all the roots are unity. Thus, we can use this property to show the Schur stability [maybe say what this is briefly].

In this paper, we will study a class of monic-palindromic polynomials that we call ``staircase palindromic polynomials". For simplicity, we will refer to a staircase palindromic polynomial as a SP polynomial. Finding the roots of any given polynomial with higher degree is generally a difficult task.  
However, in this  paper we will present a formula that gives all the irreducible factors of a given SP polynomial. Moreover, zeros of these factors are roots of unity.  Because of the symmetric property of coefficients of SP polynomials, we observe that these polynomials can be written as the product  of cyclotomic polynomials and hence all the roots of the SP polynomial are zeros of cyclotomic polynomials.
In addition to SP polynomials, we will present some special types of polynomials that behave like SP polynomials. 

Interestingly, SP polynomials arise in the study of a certain class of periodic solutions found in a population model of the cell cycle~\cite{RAT} (see \cite{YKR} for background information).
They appear as part of the characteristic polynomial of the linearization about the periodic solutions. A polynomial is said to be Schur stable if all its roots lie in the open unit disk. A periodic orbit is asymptotically stable if the corresponding characteristic polynomial is Schur stable. So, the analysis of SP polynomials plays an important role in determining the 
stability of the periodic solution of the model.

\footnotetext[1]{Mathematics, Ohio University, Athens, OH USA}
\footnotetext[2]{Corresponding author, {\tt rk733813@ohio.edu }}

% %%%%%%%%%%%%%%%%%%%%%%%%%%%%%%%%%%%%%%%%%%%%%%%%%%%%%%%%%%%%%%%%%%
% %%%%%%%%%%%%%%%%%%%%%%%%%%%%%%%%%%%%%%%%%%%%%%%%%%%%%%%%%%%%%%%%%%
% %%%%%%%%%%%%%%%%%%%%%%%%%%%%%%%%%%%%%%%%%%%%%%%%%%%%%%%%%%%%%%%%%%
% %%%%%%%%%%%%%%%%%%%%%%%%%%%%%%%%%%%%%%%%%%%%%%%%%%%%%%%%%%%%%%%%%%

\section{Basic results and notations}

First we will go over some notation, few standard definitions and results.

%%%%%%%%%%%%%%%%%%%%%%%%%%%%%%%%%%%%%%%%%%%%%%%%%%%%%%%%
%\smallskip

\subsection{Roots of unity}
A~$n^{th}$ root of unity is a number~$x$ satisfying the equation~$x^n - 1= 0$. For given~$n$, the distinct roots of unity for~$x^n - 1$ can be written as
 $$
 1, \zeta, \zeta^2, \ldots, \zeta^{n-1},
 $$
%and moreover, 
% $$
% x^n- 1 = (x -1)(x -\zeta)(x -\zeta^2)\ldots (x-\zeta^{n-1}),
% $$
where $ \zeta = \cos(2\pi/n) + i \sin(2\pi/n) = e^{\frac{2\pi\cdot 1}{n}}$. These roots are equally spaced
on the unit circle and are the vertices of a regular polygon, called the unit n-gon.

\smallskip

% \begin{defn}
% %(Primitive root of unity) 
% An $n^{th}$ root of unity $\zeta$ is called a {\em primitive $n^{th}$ root of unity} 
% if the numbers $1, \zeta, \zeta^2, \ldots, \zeta^{n-1}$ are all distinct.
% \end{defn}
% {\ty Is this definition used?}
% \smallskip

Next we consider the cyclotomic polynomials, which are closely associated with roots of unity.

\smallskip

%\begin{defn}
% A polynomial $p(x)$ is called {\em irreducible} if~$p(x)$ is a non-constant polynomial over a field~$F$ whose coefficients belong to~$F$ and cannot be factored into the product of two non-constant polynomial with coefficients in~$F$.
%\end{defn}

\smallskip

\begin{defn}%(Cyclotomic polynomial) 
The {\em $n^{th}$ cyclotomic polynomial}~$\Psi_{n}(x)$ for any positive integer~$n$, is the unique irreducible polynomial with integer coefficients that is a divisor of~$x^n - 1$ and is not a divisor of~$x^k - 1$ for any $k<n$.
\end{defn}

Note: For~$n>2$, the cyclotomic polynomials are monic-polynomials with integer coefficients that are irreducible over the field of the rational numbers.
\smallskip

\begin{lem}
If $k$ is prime, then the cyclotomic polynomial of $x^k-1$ is 
$$
\Psi_k(x) = x^{k-1}+ x^{k-2} + \ldots + x + 1. 
$$ 
\end{lem}

\smallskip

\begin{thm}\label{th:factor}
Let $n$ be a positive integer, then 
$$
x^n-1 = \displaystyle \prod_{d|n} \Psi_{d}(x). 
$$ 
\end{thm}

\subsection{Some elementary definitions}
Consider a polynomial $$P(x) = a_n x^n + a_{n-1} x^{n-1}+ \cdots + a_1 x + a_0$$  of degree~$n$, where ~$n\geq 1 ~\textnormal{and}~ a_i \in \mathbb{R}$.

\smallskip

\begin{defn} %(monic-polynomial)
The polynomial $P(x)$ is called {\em monic-polynomial} if~$a_n = 1.$
\end{defn}

\smallskip

\begin{defn}%(All one polynomial)
The polynomial $P(x)$ is called all {\em all-one polynomial} if $a_i = 1$ for all~$i$.
\end{defn}
\smallskip

\begin{defn}%(Palindromic polynomial) 
The polynomial $P(x)$ is called {\em palindromic} if $a_i = a_{n-i}$ for every~$i =~0, 1, \ldots, n.$
\end{defn}

\medskip

%%%%%%%%%%%%%%%%%%%%%%%%%%%%%%%%%%%%%%%%%%%%%%%%%%%%%%%%
%%%%%%%%%%%%%%%%%%%%%%%%%%%%%%%%%%%%%%%%%%%%%%%%%%%%%%%%
%%%%%%%%%%%%%%%%%%%%%%%%%%%%%%%%%%%%%%%%%%%%%%%%%%%%%%%%
%%%%%%%%%%%%%%%%%%%%%%%%%%%%%%%%%%%%%%%%%%%%%%%%%%%%%%%%

\section{Staircase palindromic polynomials}

\begin{defn}
For given $n$, consider monic-palindrome polynomials $$P(x) = a_n x^n + a_{n-1} x^{n-1}+ \cdots + a_1 x + a_0$$
of degree~$n$ and let $h$ be a positive integer satisfying $1 \leq h \leq \lceil{\frac{n+1}{2}\rceil}$. For fixed~$h$, if the coefficients of~$P(x)$ satisfy the following conditions:
\begin{align*}
    a_{\ell} &= a_{n-\ell} = \ell+1 ~\textnormal{for} ~\ell = 0, 1, \ldots h-1 ~\textnormal{and}\\ a_{h-1} &= a_{h} = \ldots = a_{n-h+1} = h.
\end{align*} 
Then the class of such polynomials~$P(x)$of degree~$n$ will be called a {\em staircase palindromic polynomials}.

For fixed height $h$ an staircase palindromic polynomial of degree~$n$ will be called SP polynomial and is  denoted by $S(x,n,h)$.
\end{defn}
Note: For any positive integer $n$, there are $\lceil \frac{n+1}{2}\rceil$ SP polynomials of degree n.

\begin{ex}
For~$n =7$, the possible values of the height ~$h$ are~$1, 2, 3$ and~$4$. The following polynomials are all SP polynomials of degree~$7$:
\begin{align*}
S(x,1,7) 
    &=x^7+x^6+x^5+x^4+x^3+x^2+x+1.\\
S(x,2,7)
    &=x^7+2x^6+2x^5+2x^4+2x^3+2x^2+2x+1.\\
S(x,3,7)
    &=x^7+2x^6+3x^5+3x^4+3x^3+3x^2+2x+1.\\
S(x,4,7)
    &=x^7+2x^6+3x^5+4x^4+4x^3+3x^2+2x+1.
\end{align*}
That is, the staircase palindromic polynomials of degree~$7$ is given by $\{S(x,1,7), S(x,2,7), S(x,3,7), S(x,4,7)\}$.
\end{ex}

%\begin{figure}[!ht]\centering
%\begin{tikzpicture}
%\begin{axis}[%
% scale only axis,
% xmin=0, xmax=8.9,
%   xticklabels={$a_0$,$a_1$,$a_2$,$a_3$,$a_4$,$a_5$,$a_6$,$a_7$},xtick={1,...,8},
% %xlabel={$\text{Debiet [m}^\text{3}\text{/h]}$},
% xmajorgrids,
% ymin=0, ymax=5.9,
% ymajorgrids,
% axis lines=left,
% %title={ },
% legend style={nodes=right}]
% \addplot [black, no markers, thick] coordinates {(1,1)
%         (8,1)};
% \addplot [red, no markers, thick] coordinates {(2,2)
%         (7,2)};
% \addplot [red, no markers, thick] coordinates {(2,2)
%         (2,1)};
% \addplot [red, no markers, thick] coordinates {(7,2)
%         (7,1)};
% \addplot [blue, no markers, thick] coordinates {(3,3)
%         (3,2)};
% \addplot [blue, no markers, thick] coordinates {(6,3)
%         (6,2)};        
% \addplot [blue, no markers, thick] coordinates {(3,3)
%         (6,3)};
% \addplot [green, no markers, thick] coordinates {(4,4)
%         (4,3)};
% \addplot [green, no markers, thick] coordinates {(5,4)
%         (5,3)};        
% \addplot [green, no markers, thick] coordinates {(4,4)
%         (5,4)};

% %\end{axis}
% \node at (8,1.5) {S(x,7,1)};
% \node at (7,2.7) {S(x,7,2)};
% \node at (6,3.9) {S(x,7,3)};
% \node at (5,5.2) {S(x,7,4)};
%\end{tikzpicture}
%\caption{Figure showing distribution of coefficients of staircase polynomials of degree $7$.}
%\end{figure}

\begin{lem}\label{lem:all1}
Consider an SP polynomial 
$$
S(x,n,h) = \displaystyle \sum_{i=1}^{h-1} i x^{i-1} + h \sum_{j=h-1}^{n+1-h}  x^{j} + \sum_{i=1}^{h} (h-i) x^{i+n}.
$$ 
Then there exist~$\ell, m$ such that~$S(x,n,h)$ is equal to the product of two all-one polynomials, i.e.,
$$
\displaystyle S(x,n,h)=\sum_{i=1}^{l} x^i \cdot \sum_{j=1}^{m} x^j. 
$$ 
 
\end{lem}

\begin{proof}
Consider the product 
\begin{align*}
\displaystyle \sum_{i=1}^{\ell} x^i \cdot \sum_{j=1}^{m} x^j
&= (1+x+x^2+\ldots+x^{\ell}) \cdot (1+x+x^2+\ldots+x^{m})
%&=\opadd{1+x +x^2 + \ldots+ x^{n-h} + x^{n+1-d}}{\quad + x +x^2 + \ldots + x^{n-d} + x^{n+1-d} + x^{n+2-d}} \\
%&= \quad \quad \quad \quad + x^2 + \ldots + x^{n-d}  + x^{n-d+1} + x^{n-d+2} + x^{n-d+3}\\
%& \hspace{2.5cm} \quad \vdots \hspace{2.5cm} \vdots \\
% & \hspace{2.5cm} \quad \quad + x^{n-d} + x^{n+1-d}  + \ldots + x^{n-1}\\
% & \hspace{2.8cm} \quad + x^{d-1} + x^{d}  + \ldots + x^{n-1} + x^{n}\\
% &= 1 + 2x + 3x^2 + \ldots + d x^{d-1} + \ldots +d x^{n-d} + \ldots + 3 x^{n-2} + 2x^{n-1} + x^n
\end{align*}

\[
=
\renewcommand\arraystretch{1.2}
\begin{array}{c@{}c@{\,}c@{\,}c@{\,}c@{\,}c@{\,}c@{\,}c@{\,}c@{\,}c@{\,}ccl}
& \makebox[0pt]{\raisebox{-.5\normalbaselineskip}[0pt][0pt]{}} 
        1+&x+&x^2+&\ldots+&x^{\ell-1}+&x^{\ell}&~ &~ &~ &~ & ~& \\ %9
\quad & ~&x+&x^2+&\ldots+&x^{\ell-1}+&x^{\ell}+&x^{\ell+1}&~ &~ &~ &\\ %9
\qquad &~ &~ &x^2+&\ldots+&x^{\ell-1}+& x^{\ell}+&x^{\ell+1}+&x^{\ell+2}&~& ~& \\
& & & & & & \vdots & & & & &\\
&~ &~  & ~ & ~ & x^{m-1} +&  x^{m} + & x^{m+1} + & \ldots  &  + x^{\ell+m-2} & + x^{\ell+m-1} &  \\
\qquad &~ & ~ & ~ & ~ & ~&  x^{m}+& x^{m+1} + & \ldots  & + x^{\ell+m-2} & + x^{\ell+m-1}&+ x^{\ell+m} &   \\
\cline{1-12}\\[-1.09\normalbaselineskip]
\end{array}
\]
\begin{align*}
&=\displaystyle \sum_{i=1}^\ell ix^{i-1} + (\ell + 1) \sum_{j=\ell}^m x^j + \sum_{i=1}^\ell (\ell+1-i)x^{m+i}\\
&=\displaystyle \sum_{i=1}^{h-1} ix^{i-1} + h \sum_{j=h-1}^{n+1-h} x^j + \sum_{i=1}^{h-1} (h-i)x^{(n+1-h)+i},
\end{align*} 
where $\ell = h-1$ and 
$m+\ell=n$.
\end{proof}

\begin{thm}\label{thm:sp}
Consider an SP polynomial $S(x,n,h)$, with~$h>1$,
\begin{equation}\label{eq:mainSP}
S(x,n,h) = \prod_{\substack{{\delta|h}\\\delta \neq 1}} \Psi_\delta(x) \cdot \prod_{\substack{\tau|n+2-h\\\tau \neq 1}} \Psi_\tau(x),
\end{equation}
where  $\delta$ and~$\tau$  range over the set of divisors of~$h$ and~$n+2-h$, respectively. 
And  \begin{equation}\label{eq:mainSP1}
S(x,n,1) = \prod_{\substack{\tau_1|n+1\\ \tau_1 \neq 1}} \Psi_{\tau_1}(x),
\end{equation}
where~$\tau_1$ ranges over the set of divisors of~$n+1$.
The symbol~$\Psi_i(z)$ denotes the~$i^{th}$~cyclotomic polynomial.  
\end{thm} 

\begin{proof}
If  $S(x,n,h)$ is an SP polynomial, we have 
$$a_n = a_0 = 1, a_{n-1} = a_1 = 2, \ldots, a_{n-h+1} = a_{h-1} = h,$$ that is,
\begin{equation*}
    S(x,n,h) 
        = 1 + 2x + \cdots + h x^{h-1}+\cdots+ hx^{n+1-h}+\cdots+2x^{n-1}+x^n.
\end{equation*}

Then, using Lemma~\ref{lem:all1} we have
\begin{equation}\label{eq:all1SP}
    S(x,n,h) 
        =(1+\cdots+x^{h-1}) \cdot (1+\cdots+x^{n-h+1}).
\end{equation}
Since all the roots of all-one polynomials are roots of unity other than unity itself, so
$$
S(x,n,h)=\frac{x^h-1}{x - 1} \cdot \frac{x^{n+2-h}-1}{x - 1}.
$$
Now from Theorem \ref{th:factor},
\begin{align*}
    %&= \frac{\displaystyle \prod_{\delta|d} \Psi_{\delta}(x)}{x-1} \cdot \frac{\displaystyle \prod_{\tau|n+2-d} \Psi_{\tau}(x)}{x-1}\\
    &= \frac{\displaystyle \prod_{\delta|h} \Psi_{\delta}(x)}{x-1} \cdot \frac{\displaystyle \prod_{\tau|n+2-h} \Psi_{\tau}(x)}{x-1}\\
    &= \displaystyle \prod_{\substack{{\delta}|h\\{\delta} \neq 1}} \Psi_{\delta}(x) \cdot  \prod_{\substack{{\tau}|n+2-h\\{\tau} \neq 1}} \Psi_{\tau}(x),\\
    % &= \displaystyle \prod_{\substack{{\delta}|h}} \Psi_{\delta}(x) \cdot  \prod_{\substack{{\tau}|n+2-h}} \Psi_{\tau}(x),
\end{align*}
where $\delta $ and $\tau$ are divisors greater than~$1$ of $h$ and $n+2-h$ respectively.

Next for~$h=1$, the first all one polynomial on the right hand side of Equation~\ref{eq:all1SP} is equal to~$1$ and $n+2-h = n+1$. So we have 
$$
S(x,n,1) =  \prod_{\substack{\tau_1|n+1\\ \tau_1 \neq 1}} \Psi_{\tau_1}(x),
$$
where~$\tau_1$ is an element from the set of divisors of~$n+1$.
\end{proof}

\begin{ex}
Consider an SP polynomial 
$$
S(x,4,3)=x^4+2x^3+3x^2+2x+1
$$ of degree $4$. Then the  divisor of both~$h = 3$ and~$n+2-h = 4+2-3 =3$ is $3$, so 
\begin{align*}
  S(x,4,3) &=\Psi_3(x)\cdot\Psi_3(x)\\
  &=(x^2+x+1)^2.  
\end{align*}
\end{ex}

%{\ty The following cannot be a corollary, b/c it is well known and does not follow from your results, it follows from more basic considerations.}
\begin{rem}
If~$\zeta$ is a root of the polynomial~$S(x,n,h)$, then the 
conjugate~$\overline{\zeta}$ is also a root of~$S(x,n,h)$.
Note that the only possible real root of an SP polynomials is~$-1$.
%\begin{cor}
Also, if $\zeta = -1$ is the root of an SP polynomial $S(x,n,h)$, 
then either~$n$ is odd or~$n$ is even and~$h=2$. 
\end{rem}
%\end{cor}

%%%%%%%%%%%%%%%%%%%%%%%%%%%%%%%%%%%%%%%%%%%%%%%%%%%%%%%%%%
%%%%%%%%%%%%%%%%%%%%%%%%%%%%%%%%%%%%%%%%%%%%%%%%%%%%%%%%%%
%%%%%%%%%%%%%%%%%%%%%%%%%%%%%%%%%%%%%%%%%%%%%%%%%%%%%%%%%%
%%%%%%%%%%%%%%%%%%%%%%%%%%%%%%%%%%%%%%%%%%%%%%%%%%%%%%%%%%

\section{Extensions to related polynomials }

In this section we will derive some extensions of SP polynomial which are motivated by following properties of cyclotomic polynomials: 
\subsection{Some properties of cyclotomic polynomials}

\begin{enumerate}
\item For odd number $m>1$ , $\Psi_m(-x)=\Psi_{2m}(x).$
\item If  $p$ is a prime and $m$ is a positive integer. If $p$ divides $m$, then $\Psi_{mp}(x)=\Psi_{m}(x^p).$
\item If $p$ is prime and~$m$ is a positive integer. If $p$ does not divide $m$, then 
$$\Psi_m(x^p) = \Psi_{pm}(x)\Psi_m(x)
$$ 
\end{enumerate}
\subsection{Alternating sign SP polynomials (ASP)}

% \begin{figure}[ht!]\centering
% \begin{tikzpicture}
% \begin{axis}[%
% scale only axis,
% xmin=0, xmax=9.9,
%   xticklabels={$a_0$,$a_1$,$a_2$,$a_3$,$a_4$,$a_5$,$a_6$,$a_7$,$a_8$},xtick={1,...,9},
% %xlabel={$\text{Debiet [m}^\text{3}\text{/h]}$},
% xmajorgrids,
% ymin=-4.9, ymax=5.9,
% ymajorgrids,
% axis lines=left,
% %title={ },
% legend style={nodes=right}]
% \addplot [red, no markers, thick] coordinates {(1,1)
%         (2,-2)};
% \addplot [red, no markers, thick] coordinates {(2,-2)
%         (3,3)};
% \addplot [red, no markers, thick] coordinates {(3,3)
%         (4,-4)};
% \addplot [red,no markers, thick] coordinates {(4,-4)
%         (5,5)};
% \addplot [blue, no markers, thick] coordinates {(5,5)
%         (6,-4)};
% \addplot [blue, no markers, thick] coordinates {(6,-4)
%         (7,3)};        
% \addplot [blue, no markers, thick] coordinates {(7,3)
%         (8,-2)};
% \addplot [blue, no markers, thick] coordinates {(8,-2)
%         (9,1)};
% %\addplot [blue, no markers, thick] coordinates {(5,4)        (5,3)};        
% %\addplot [blue, no markers, thick] coordinates {(4,4)         (5,4)};

% \end{axis}
% %\node at (1,1.1);
% %\node at (7,2.7);
% %\node at (6,3.9);
% %\node at (5,5.2);
% %
% \end{tikzpicture}
% %\caption{Figure showing distribution of coefficients of alternating SP polynomial of degree~$8$.}
% \end{figure}

\begin{defn}
Consider a polynomial
$$P(x)=(-1)^n x^n+(-1)^{n-1}a_{n-1}x^{n-1}+ ...+(-1)a_1x+a_0$$
with degree $n$. If the coefficients $a_i$ satisfy 
the conditions of the coefficients of a SP polynomial then polynomial is called alternating sign SP polynomial (ASP) of height~$h$ and will be denoted by~$T(x, n, h)$. 
\end{defn}

Using the properties of cyclotomic polynomials and  Theorem~\ref{thm:sp}, we can derive a formula for ASP analogous to SP polynomials.
\begin{lem}\label{lem:Asp}
%Consider a monic-palindrome polynomial
%$$T(x,n,h)=(-1)^na_nx^n+(-1)^{n-1}a_{n-1}x^{n-1}+ ...+(-1)a_1x+a_0$$
%with even degree n,  where $a_i$ are  positive integers that satisfy 
%the conditions of the coefficients of a staircase palindromic polynomial. 
%Let $h = a_\frac{n+1}{2}$, Then
Let $T(x,n,h)$ be an $ASP$, then
$$
    T(x,n,h) = \prod_{\substack{\delta|h\\ \delta \neq 1}} \Psi_\delta(-x) \cdot \prod_{\substack{\tau|n+2-h\\ \tau \neq 1}} \Psi_\tau(-x),
$$
 where $\delta$ and $\tau$ are divisors  of $h$  and $n+2-h$ respectively.
 And  \begin{equation*}
T(x,n,1) = \prod_{\substack{\tau_1|n+1\\ \tau_1 \neq 1}} \Psi_{\tau_1}(-x),
\end{equation*}
where~$\tau_1$ ranges over the set of divisors of~$n+1$.
\end{lem} 

\begin{proof}
Let $y=-x.$ Then  $T(x,n,h)=S(-y,n,h)$ is a SP polynomial. Now using  Theorem~\ref{thm:sp},   
\begin{align*}
    T(x,n,h) 
             &= \prod_{\substack{\delta|h\\ \delta \neq 1}} \Psi_\delta(y) \cdot \prod_{\substack{\tau|n+2-h\\\tau\neq 1}} \Psi_\tau(y)\\
             &= \prod_{\substack{\delta|h\\\delta \neq 1}} \Psi_\delta(-x) \cdot \prod_{\substack{\tau|n+2-h\\ \tau \neq 1}} \Psi_\tau(-x).
\end{align*}
Therefore, 
$$ 
T(x,n,h) 
   = \prod_{\substack{\delta|h\\ \delta \neq 1}} \Psi_\delta(-x) 
     \cdot \prod_{\substack{\tau|n+2-h\\ \tau \neq 1}} \Psi_\tau(-x),
$$
where~$\Psi_i(-x)$ are cyclotomic polynomials obtained by plugging~$-x$ in the~$i^{th}-$cyclotomic polynomials.
Similarly, for~$h=1$ using Theorem~\ref{thm:sp} we get
\begin{equation*}
T(x,n,1) = \prod_{\substack{\tau_1|n+1\\ \tau_1 \neq 1}} \Psi_{\tau_1}(-x).
\end{equation*}
\end{proof}

% \begin{ex}
% Consider ASP Polynomial $$T(x,8,5)=x^8-2x^7+3x^6-4x^5+5x^4-4x^3+3x^2-2x+1.$$ %which can be written as ~$S(-y,8,5)$, by letting ~$x=-y$~, then $S(-y,8,5)$  is SP Polynomial. 
% Here, the divisor of both~$h = 5$ and~$n+2-h = 8+2-5 =5$ is $5$ then by  Lemma~\ref{lem:Asp} we have 
% \begin{align*}
%   T(x,8,5)  &= (\Psi_{10}(x))^2\\
%          & =(x^4-x^3+x^2-x+1)^2.  
% \end{align*}
% \end{ex}

%One of the properties of  cyclotomic polynomials is that for odd number $m$ , $\Psi_m(-x)=\Psi_{2m}(x)$ which can be used to simplify the formula above as the follow:

\subsection{Missing terms  SP polynomials (MSP) }

% {\ty I don't thing "missing term" is a good name, but leave it for now and ask Henry and Lopez if they can think of a better one.}
\begin{defn}
 Consider a polynomial
$$
    P(x)
        =x^{n}+a_{(n-d)}x^{n-d}+...+a_{d}x^{d}+1
$$ 
of degree~$n$. If $d$ is a divisor of $n$  and the coefficients of~$P(x)$ satisfy the conditions of an SP polynomial of degree~$n$ and height~$h=a_{\lceil \frac{n+1}{2}\rceil}$ then the polynomial will be called missing terms SP polynomial (MSP) and is denoted by~$M(x, n, h).$
\end{defn}

\begin{lem}\label{lem:msp}
Let~$M(x,n,h)$ be an MSP polynomial, then\\
$$
    M(x,n,h) 
        = \prod_{\substack{\delta|h\\ \delta \neq 1}} \Psi_\delta(x^{d}) \cdot \prod_{\substack{\tau|\frac{n}{d}+2-h\\ \tau \neq 1}} \Psi_\tau(x^{d}).
$$
where $\delta$ and $\tau$ are divisors  of $h$  and $n+2-h$ respectively.
 And  \begin{equation*}
M(x,n,1) = \prod_{\substack{\tau_1|n+1\\ \tau_1 \neq 1}} \Psi_{\tau_1}(x^d),
\end{equation*}
where~$\tau_1$ ranges over the set of divisors of~$n+1$.
\end{lem}

\begin{proof}
Let $y=x^d.$ Since~$d | n$ so we know  $M(x,n,h)=S(y,n/d,h)$ is an SP polynomial. Now using  Theorem~\ref{thm:sp},   
\begin{align*}
    M(x,n,h) &= S(y, n/d, h)\\
             &= \prod_{\substack{\delta|h\\ \delta \neq 1}} \Psi_\delta(y) \cdot \prod_{\substack{\tau|\frac{n}{d}+2-h\\\tau\neq 1}} \Psi_\tau(y)\\
             &= \prod_{\substack{\delta|h\\\delta \neq 1}} \Psi_\delta(x^d) \cdot \prod_{\substack{\tau|\frac{n}{d}+2-h\\ \tau \neq 1}} \Psi_\tau(x^d).
\end{align*}
Therefore, 
$$ 
M(x,n,h) 
   = \prod_{\substack{\delta|h\\ \delta \neq 1}} \Psi_\delta(x^d) 
   \cdot \prod_{\substack{\tau|\frac{n}{d}+2-h\\ \tau \neq 1}} \Psi_\tau(x^d),
$$
where~$\Psi_i(x^d)$ are  polynomials obtained by plugging~$x^d$ in the~$i^{th}-$cyclotomic polynomials.
Similarly, for~$h=1$ using Theorem~\ref{thm:sp} we get
\begin{equation*}
M(x,n,1) = \prod_{\substack{\tau_1|n+1\\ \tau_1 \neq 1}} \Psi_{\tau_1}(x^d).
\end{equation*}
\end{proof}

\subsection{Geometric SP polynomials (GSP)}
Next, we will look at polynomials whose coefficients are not palindromic but can be made palindromic by a particular substitution, such polynomials will be called a geometric SP polynomial.

%In some cases, the given polynomial is not SP. However, we can use  $y=\alpha{x}$  to convert the given polynomial to be in the  SP form by choosing the number $\alpha.$ 
\begin{defn}
Consider a polynomial
$$
G(x)
   =   b_n x^{n}+b_{n-1}x^{n-1}+...+b_{1}x+1$$ of degree~$n$. If $a_i=\frac{b_i}{\alpha^i}$ are coefficients of an SP polynomial,
where $\alpha={b_n}^\frac{1}{n}$,  then the polynomial~$G(x)$ will be called a geometric SP polynomial (GSP). 
\end{defn}

\begin{rem}
Note, if $\alpha=1$ then GSP is an SP polynomial and if $\alpha=-1$ then GSP is an ASP polynomial.
\end{rem}

\begin{lem}\label{lem:gsp} 
 Let $G(x)$ be a GSP polynomial of degree~$n$, then
$$ G(x) = \prod_{\substack{\delta|h\\ \delta \neq 1}} \Psi_\delta(\alpha{x}) 
   \cdot \prod_{\substack{\tau|n+2-h\\ \tau \neq 1}} \Psi_\tau(\alpha{x}).$$ where $h=a_{\lceil \frac{n+1}{2}\rceil}>1.$ 
And if~$h=1$ \begin{equation*}
G(x) = \prod_{\substack{\tau_1|n+1\\ \tau_1 \neq 1}} \Psi_{\tau_1}(\alpha x),
\end{equation*}
where~$\tau_1$ ranges over the set of divisors of~$n+1$.
   \end{lem}
\begin{proof}

Let $y=\alpha{x}.$  Then   
\begin{align*}
    G(x) &= b_n x^{n}+b_{n-1}x^{n-1}+...+b_{1}x+1.\\
    &= a_n (\alpha x)^{n}+a_{n-1} (\alpha x)^{n-1}+...+a_{1}(\alpha x)+1.\\
    &= S(y, n,h).
\end{align*}
Now using  Theorem~\ref{thm:sp}, we have
\begin{align*}
     G(x)  &= S(y, n,h)\\
     &= \prod_{\substack{\delta|h\\ \delta \neq 1}} \Psi_\delta(y) \cdot \prod_{\substack{\tau|n+2-h\\\tau\neq 1}} \Psi_\tau(y)\\
             &= \prod_{\substack{\delta|h\\\delta \neq 1}} \Psi_\delta(\alpha{x}) \cdot \prod_{\substack{\tau|n+2-h\\ \tau \neq 1}} \Psi_\tau(\alpha{x}).
\end{align*}
Therefore, 
$$ G(x) 
   = \prod_{\substack{\delta|h\\ \delta \neq 1}} \Psi_\delta(\alpha{x}) 
   \cdot \prod_{\substack{\tau|n+2-h\\ \tau \neq 1}} \Psi_\tau(\alpha{x}),$$
where~$\Psi_i(\alpha{x})$ are  polynomials obtained by plugging~$\alpha{x}$ in the~$i^{th}$cyclotomic polynomials.
Similarly, for~$h=1$ using Theorem~\ref{thm:sp} we get
\begin{equation*}
G(x) = \prod_{\substack{\tau_1|n+1\\ \tau_1 \neq 1}} \Psi_{\tau_1}(\alpha x).
\end{equation*}
\end{proof}

\begin{figure}[ht!]
\centering
\begin{tikzpicture}
%\begin{axis}[%
%scale only axis,
%xmin=0, xmax=7,
%  xticklabels={$a_0$,$a_1$,$a_2$,$a_3$,$a_4$,$a_5$ ,$a_6$},xtick={1,...,9},
%xlabel={$\text{Debiet [m}^\text{3}\text{/h]}$},
%xmajorgrids,
%ymin=.001, ymax=1,
%ymajorgrids,
%axis lines=left,
%title={ },
%legend style={nodes=right}]
%\addplot [red, no markers, thick] coordinates {(1,1)
  %      (2,1)};
%\addplot [red, no markers, thick] coordinates {(2,1) (2,.333)};
%\addplot [red, no markers, thick] coordinates {(2,.333) (3,.333)};
%\addplot [red, no markers, thick] coordinates {(3,.333) (3,.1)};
%\addplot [red, no markers, thick] coordinates {(3,.1) (4,.1)};
%\addplot [red,no markers, thick] coordinates {(4,.1) (4,.03)};

%\addplot [red, no markers, thick] coordinates {(4,.03) (5,.03)};

%\addplot [red, no markers, thick] coordinates {(5,.03) (5,.01)}; 
%\addplot [red, no markers, thick] coordinates {(5,.01) (6,.01)};
%\addplot [red, no markers, thick] coordinates {(6,.01) (6.004)};
%\addplot [red, no markers, thick] coordinates {(6,.004) (7,.004)};
%\end{axis}
%\node at (1,1.1);
%\node at (7,2.7);
%\node at (6,3.9);
%\node at (5,5.2);
%
\end{tikzpicture}
%\caption{Figure showing distribution of coefficients of GSP of degree~$7$ when $\alpha=\frac{1}{2}.$}
\end{figure}

%%%%%%%%%%%%%%%%%%%%%%%%%%%%%%%%%%%%%%%%%%%%%%%%%%%%%%%%

\section{Acknowledgement}

%%%%%%%%%%%%%%%%%%%%%%%%%%%%%%%%%%%%%%%%%%%%%%%%%%%%%%%%

\end{document}